\documentclass{article}

\usepackage{color}
\usepackage{amsmath, amssymb, amsthm}

\newtheorem{theorem}{Theorem}
\newtheorem{lemma}{Lemma} 
\newtheorem{prop}{Proposition}

\newtheorem*{proof*}{Proof}

\allowdisplaybreaks

\begin{document}

\title{{\bf Number of spanning trees in a wheel graph \\ with two identified vertices via hitting times}}

\author{Shunya TAMURA \\
Okegawa City Okegawa West Junior High School \\
Kawataya, Okegawa, 363-0027, Japan \\ 
e-mail: shunya.tamura059@gmail.com \\
%Tel.: +81-45-339-4205, Fax: +81-45-339-4205 \\  \\ 
\\ 
Yuuho TANAKA \\
Faculty of Fundamental Science and Engineering \\ 
Waseda University \\
Shinjuku, Tokyo, 169-8555, Japan \\
e-mail: tanaka\_yuuho@aoni.waseda.jp
}

\date{\empty}

\maketitle
\begin{abstract}
In this paper, we provide an exact formula for the average hitting times in a wheel graph $W_{N+1}$ using a combinatorial approach.
For this wheel graph, the average hitting times can be expressed using Fibonacci numbers when the number of surrounding vertices is odd and Lucas numbers when it is even.
Furthermore, combining the exact formula for the average hitting times with the general formula for the effective resistance of the graph allows determination of the number of spanning trees of the graph with two identified vertices. 

\end{abstract}

\vspace{10mm}

\begin{small}
\par\noindent
{\bf Keywords}: random walk, hitting time, spanning tree, Fibonacci number, Lucas number 
\end{small}
\begin{small}
\par\noindent
{\bf Abbr. title}: Average hitting time in a wheel graph
\par\noindent
{\bf 2000 Mathematical Subject Classification}: 60F05, 05C50, 15A15. 
\end{small}

\vspace{10mm}

%%%%%%%%%%%%%%%%%%%%%%%%%%%%%%%%%%%%%%%%%%%%%%%%%%%%%%%%%%%%%%%%%%%%%%%%%%%%%%%%%%%%%%%%%%%%%%%%%%%%%%%%%%%%%%%%%%%%%%%%%%%%%%%%%%%%%%%%%%%%%%%%%%%%%%%%%%
\section{Introduction \label{sec01}}

A simple random walk on a graph $G$ is a discrete stochastic process in which a particle initially located at a vertex $x$ moves to an adjacent vertex $y$ along the edge $xy$ with a probability of $1/d(x)$, where $d(x)$ is the degree of vertex $x$ of $G$.
Random walks offer a concise and powerful means of modeling stochastic processes, particularly for understanding probabilistic phenomena that evolve over time; because each step of a random walk is probabilistic and independent, they are effective for describing many real-world phenomena.
Also, random walks are important in graph theory~\cite{Lovasz}.

Furthermore, for two vertices $x$ and $y$ on graph $G$, the expected number of steps for a random walk starting from $x$ to reach $y$ for the first time is called the \emph{average hitting time from $x$ to $y$} and is denoted as $h(G; x, y)$. 
The average hitting time is an important metric in fields such as probability theory, graph theory, and network analysis~\cite{Ald, Sharavas}. For example, it can be used to evaluate the accessibility and efficiency of specific vertices within a graph, and it is important in the analysis of random movements including diffusion phenomena. 
However, the analysis of the average hitting time on a graph depends heavily on the rules of the random walk and the structure of the graph, so it is difficult to find an explicit formula for the average hitting time from one vertex to any other vertex in a general graph~\cite{Ald, Sharavas}.
On the other hand, explicit formulas have been provided for classes of highly symmetric graphs and similar structures~\cite{Etal, Nishimura, TS, TY, TY2}.
In particular, Yang~\cite{YY} obtained an exact formula for the average hitting times of simple random walks on a wheel graph $W_{N+1}$ formed by placing a single vertex inside an $N$-vertex cycle and connecting the vertices of the cycle to the central vertex.
Note that $W_{N+1}$ is assumed to be a simple graph containing no multiple edges or loop edges. 

In this paper, we obtain an exact formula for the average hitting times of simple random walks on the aforementioned wheel graph by using an analytical method similar to that studied in \cite{Etal, TS, TY, TY2}.
Also, by using the general formula for the average hitting time and the effective resistance of the graph, we provide an exact formula for the number of spanning trees in the graph with two identified vertices.
The number of spanning trees in a graph is a widely used metric for gaining deeper insight into the graph's structure and properties, with many practical applications.
Specifically, it suggests the connectivity and redundancy of the graph, information that is valuable for network design and fault tolerance analysis. 
For the wheel graph $W_{N+1}$, formulas for metrics such as the average hitting time and graph resistance differ depending on the parity of the number of cycle vertices.
These formulas are expressed in terms of Fibonacci numbers for graphs with an odd number of vertices and Lucas numbers for graphs with an even number of vertices.
Our proof is based on combinatorial arguments without the need for spectral graph theory.
Furthermore, not only does the formula suggests a relationship between the average hitting time of a simple random walk on $W_{N+1}$ and the Fibonacci or Lucas numbers, this relationship is also evident from the intermediate results obtained during the proof.

The rest of this paper is organized as follows.
In Section~\ref{sec02}, we define Fibonacci and Lucas numbers and introduce the fundamental relationships between them that are necessary to obtain our results.
In Section~\ref{sec03}, we present the model under consideration, explain how to formulate the matrix equations, and describe our approach to solving the problem. 
In Section~\ref{sec04}, we obtain the exact formula for the average hitting times of simple random walks on a wheel graph and provide a proof, considering cases where the number of vertices in the graph is either odd or even.
In Section~\ref{sec05}, we calculate the effective resistance of the wheel graph and combine it with the exact formula for the average hitting time to determine the number of spanning trees in the wheel graph with two identified vertices.
Finally, we summarize our results in Section~\ref{sec06}.

%%%%%%%%%%%%%%%%%%%%%%%%%%%%%%%%%%%%%%%%%%%%%%%%%%%%%%%%%%%%%%%%%%%%%%%%%%%%%%%%%%%%%%%%%%%%%%%%%%%%%%%%%%%%%%%%%%%%%%%%%%%%%%%%%%%%%%%%%%%%%%%%%%%%%%%%%%
\section{ Fibonacci and Lucas numbers \label{sec02}}

In this section, we define the Fibonacci and Lucas sequences and then introduce several well-known formulas related to them. The Fibonacci sequence is defined by the recurrence relation $F_{n+2}=F_{n+1}+F_{n}$, where $F_{0} =0$, $F_{1} =1$, and $n=0, 1, 2, \cdots$. Similarly, the Lucas sequence is defined by the recurrence relation $L_{n+2}=L_{n+1}+L_{n}$, where $L_{0} =2$, $L_{1} =1$, and $n=0, 1, 2, \cdots$. It is well known that Fibonacci and Lucas numbers are closely related. For example, the relationship $L_{n}=F_{n-1}+F_{n+1}$ holds, and this and the following related formulas are used throughout this paper:

\begin{align}
F_{2n-2}-3F_{2n}+F_{2n+2}=0 \label{Fib1} \\
F_{-n}=(-1)^{n+1}F_n \label{Fib12} \\
L_{-n}=(-1)^nL_n \label{Luc11} \\
L_{2n}-2=L_n^2 \label{Luc10} \\
L_{2n-2}-3L_{2n}+L_{2n+2}=0  \label{Fib2} \\
%F_{m}F_{n-1}+F_{n}F_{m+1}=F_{m+n} \label{Fib3} \\
%F_{n+2} L_{m} -F_{n} L_{m-2}=L_{m+n} \label{formula4} \\
F_mL_n=F_{m+n}+(-1)^nF_{m-n} \label{formula5} \\
 F_{n} L_{m+1} +F_{n-1} L_{m}  =L_{m+n}  \label{formula6} \\
 F_{n}-F_{n-r}F_{n+r}=(-1)^{n-r}F_r^2 \label{Fib10}\\
F_n^2+F_{n+1}^2=F_{2n+1} \label{Fib11} \\
%\displaystyle \sum_{k=1}^{\ell} F_{2k}=F_{2\ell+1}-1\label{Fib7} \\
\displaystyle \sum_{k=1}^{n-\ell} F_{2k-1}=F_{2n-2\ell} \label{Fib8} \\
\displaystyle \sum_{k=1}^{\ell} L_{2k}=L_{2\ell+1}-1 \label{Luc9}
\end{align}

For the proofs of these formulas, please refer to appropriate literature on the Fibonacci and Lucas numbers \cite{Thomas, Koshy}.

%%%%%%%%%%%%%%%%%%%%%%%%%%%%%%%%%%%%%%%%%%%%%%%%%%%%%%%%%%%%%%
%%%%%%%%%%%%%%%%%%%%%%%%%%%%%%%%%%%%%%%%%%%%%%%%%%%%%%%%%%%%%%
\section{Our model \label{sec03}} 

Our model comprises a graph formed by placing a single vertex inside an $N$-vertex cycle and connecting the cycle vertices to the central one. This structure is generally called a wheel because of that resemblance. The total number of vertices in the wheel is $W_{N+1}$, consisting of an $N$-vertex cycle ($v_0, v_1, v_2, \ldots, v_{N-1}$) and one central vertex $v_N$. 
Here, when we refer to an odd number of vertices, we mean specifically those in the cycle, not the total number of vertices in the wheel. Similarly, when we refer to an even number of vertices, we mean that the cycle has an even number of vertices, resulting in the total number of vertices in the wheel being odd. Note that $W_{N+1}$ is assumed to be a simple graph, containing no multiple edges or loop edges. 

The Laplacian matrix 
$L=[L(i,j)]$ of the graph $W_{N+1}$ is defined as follows: 
\begin{center}
$L(i,j)=
\begin{cases}
\deg_{W_{N+1}}(i) & \text{if}\ i = j, \\
 -1 & \text{if}\ i\ \neq j \ \text{ and }\ (v_{i}, v_{j}) \in E, \\
 0 & \text{otherwise.}
\end{cases}$
\end{center}

Let $L'$ be the matrix obtained by removing the first row and the first column of the Laplacian matrix $L$. Define $\vec{h}$ as a column vector whose $k$-th element is $h(W_{N+1}; 0, k)$
($1\leq k\leq N-1$) and whose $N$-th element is $h(W_{N+1}; N, 0)$.

In the case of a random walk on $W_{N+1}$, a random walker moves to any adjacent vertex with equal probability. Therefore, for all $\ell \in \mathbb{Z}_N$, the following relationships hold:
\begin{align*}
h\left(W_{N+1} ; 0, \ell \right)
&=\displaystyle\frac{1}{3}\left(1+h\left(W_{N+1} ;-1,\ell \right)+1+h\left(W_{N+1} ;1,\ell \right) +1+h\left(W_{N+1} ; N, \ell \right)\right), \\
h\left(W_{N+1} ; N, 0 \right)
&=\displaystyle\frac{1}{N}\left(1+h\left(W_{N+1} ;1,0 \right)+1+h\left(W_{N+1} ;2,0 \right) +\cdots+1+h\left(W_{N+1} ; N-1, 0 \right)\right).
\end{align*}
%The wheel graph $W_{N+1}$, due to its apparent symmetry, seems to reflect the dihedral group $D_N$. However, it does not possess a strict group structure. It can be considered a highly symmetric graph-theoretical structure, formed by adding a central vertex to the rotational and reflective symmetry of the $n$ vertices on the outer cycle.
%Since the dihedral group $D_N$ acts on the wheel graph $W_{N+1}$, the following symmetry holds for all $\ell \in \mathbb{Z}_N$:
Hence, for all $\ell \in \mathbb{Z}_N$:

\begin{align*}
-h\left(W_{N+1} ;0,\ell -1\right)+3h\left(W_{N+1} ;0,\ell \right)-h\left(W_{N+1} ;0,\ell +1\right)&=3, \\
-h\left(W_{N+1} ;1,0\right)-h\left(W_{N+1} ;2,0\right)-\cdots-h\left(W_{N+1} ;N-1,0\right)+Nh\left(W_{N+1} ;N,0 \right)&=N.
\end{align*}
Then, we have the matrix equation 
$L' \vec{h} = \begin{bmatrix}3&3&\cdots&3&N\end{bmatrix}$. 
By combining this result with the facts that $h(W_{N+1}; 0, \ell) = h(W_{N+1}; 0, N-\ell)$ and $h(W_{N+1}; N, 0) = h(W_{N+1}; N, \ell)$ for all $\ell \in \mathbb{Z}_N$, the number of variables in the problem can be halved. Let $L'(i, j)$ denote the $(i, j)$-th element of the matrix $L'$. 
Let $H_{N+1}$ be the $(\lfloor N/2\rfloor+1) \times (\lfloor N/2\rfloor+1)$ matrix whose $(i,j)$th entry is
\[
\begin{cases}
L'(i,j) &\text{if $N$ is even,\;$j=\frac{N}{2}$,}\\
L'(N,j) &\text{if $i=\lfloor N/2\rfloor+1,\;j\neq\lfloor N/2\rfloor+1$,}\\
L'(i,N) &\text{if $i\neq\lfloor N/2\rfloor+1,\;j=\lfloor N/2\rfloor+1$,}\\
L'(N,N) &\text{if $i=j=\lfloor N/2\rfloor+1$,}\\
L'(i,j)+L'(i,N-j) &\text{otherwise.}
\end{cases}
\]
Define $\vec{h}'$ as a column vector of dimension $(\lfloor N/2\rfloor+1)$ whose $k$-th entry is $h(W_{N+1}; N,0)$ if $k=\lfloor N/2\rfloor+1$, and $h(W_{N+1}; 0,k)$ otherwise. 
Then, the equation $H_{N+1} \vec{h}'= \begin{bmatrix}3&3&\cdots&3&N\end{bmatrix}$ holds. In this problem, the average hitting times can be computed by solving this matrix equation directly, without relying on the graph's resistance or similar concepts. 

%%%%%%%%%%%%%%%%%%%%%%%%%%%%%%%%%%%%%%%%%%%%%%%%%%%%%%%%%%%%%%
%%%%%%%%%%%%%%%%%%%%%%%%%%%%%%%%%%%%%%%%%%%%%%%%%%%%%%%%%%%%%%
\section{Average hitting times \label{sec04}} 

In this section, we obtain an exact formula for the average hitting times of simple random walks on the wheel graph $W_{N+1}$.

\begin{theorem} \label{Mainthm}
For the wheel graph $W_{N+1}$, the exact formula for the average hitting time from central vertex $0$ to peripheral vertex $\ell$ $(1\leq \ell\leq N-1)$ can be expressed as follows:
\begin{align*}
h(W_{N+1};0,\ell)
=\begin{cases}
\displaystyle\frac{4N(F_{N}-F_{N-2\ell})}{F_{N-1}+F_{N+1}} & \text{if $N$ is odd,} \\
\displaystyle\frac{4N(L_{N}-L_{N-2\ell})}{L_{N-1}+L_{N+1}} & \text{if $N$ is even.}
\end{cases}
\end{align*}

For the wheel graph $W_{N+1}$, the exact formula for the average hitting time from central vertex $N$ to peripheral vertex $0$ can be expressed as follows:
\begin{align*}
h(W_{N+1};N,0)
=\begin{cases}
\displaystyle\frac{(4N-3)F_{N+1}-(4N+3)F_{N-1}}{F_{N-1}+F_{N+1}} & \text{if $N$ is odd,} \\
\displaystyle\frac{(4N-3)L_{N+1}-(4N+3)L_{N-1}}{L_{N-1}+L_{N+1}} & \text{if $N$ is even.}
\end{cases}
\end{align*}
\end{theorem}

We provide proofs for the average hitting times on $W_{N+1}$ by separating the cases of odd and even numbers of vertices.

\subsection{Odd \label{subsec041}}
First, we consider the coefficient matrix $H_{N+1}$ of a wheel $W_{N+1}$ whose cycle part consists of an odd number of vertices; here, note that the total number of vertices in the entire wheel graph is an even number. The coefficient matrix $H_{N+1}$ for the wheel graph $W_{N+1}$ is defined as follows: 
\begin{center}
$H_{N+1}(i,j)
=\begin{cases}
3 & \text{if }i=j,\; i \neq \frac{N-1}{2},\frac{N+1}{2},\\
2 & \text{if }i=j=\frac{N-1}{2},\\
N & \text{if }i=j=\frac{N+1}{2},\\
-1 & \text{if }|i-j|=1,\; (i, j) \neq (\frac{N+1}{2}, \frac{N-1}{2}),\\
 & j=\frac{N+1}{2},\;1\leq i\leq \frac{N-1}{2},\\
-2 & \text{if }i=\frac{N+1}{2},\; 1\leq j\leq \frac{N-1}{2},\\
0 & \mbox{otherwise.}
\end{cases}$
\end{center}

An important point to note is that this coefficient matrix is a square matrix of size $\frac{N+1}{2} \times \frac{N+1}{2}$. 
The last row represents the average hitting times from the center to the peripheral vertices. 
For example, the coefficient matrix $H_{9+1}$ for $W_{9+1}$ is 
\begin{align*}
H_{9+1} =\begin{bmatrix}
3&-1&0&0&-1 \\
-1&3&-1&0&-1 \\
0&-1&3&-1&-1 \\
0&0&-1&2&-1 \\
-2&-2&-2&-2&9
\end{bmatrix}.
\end{align*}
In this case, the inverse of the coefficient matrix $H_{N+1}^{-1}$ is given as follows. 
%%%%%  命題  %%%%%
\begin{prop} \label{prop1}
If $N$ is odd, then
\[
H_{N+1}^{-1}(i,j)
=\displaystyle\frac{1}{F_{N-1}+F_{N+1}}
\begin{cases}
2(F_{N}-F_{N-2i}) +F_{N-2j}(L_{2i}-2) &  \text{if } 1\leq i<j \leq \frac{N-1}{2}, \\
2(F_{N}-F_{N-2j}) +F_{N-2i}(L_{2j}-2) &  \text{if } 1\leq j\leq i \leq \frac{N-1}{2}, \\
2(F_{N}-F_{N-2j}) &  \text{if } i=\frac{N+1}{2},\ 1\leq j\leq \frac{N-1}{2},\\
F_{N}-F_{N-2i} &  \text{if } j=\frac{N+1}{2},\ 1\leq i\leq \frac{N-1}{2},\\
F_{N} &  \text{if } i=j=\frac{N+1}{2}.
\end{cases}
\]
\end{prop}
%%%%%%%%%%%%%%%%

\begin{proof}
%We use $K_{N+1}$ to stand for the matrix $K_{N+1}(i,j)$. 
In this case the matrix $K_{N+1}$ is defined as follows: 
\[
K_{N+1}(i,j)
:=\displaystyle\frac{1}{F_{N-1}+F_{N+1}}
\begin{cases}
2(F_{N}-F_{N-2i}) +F_{N-2j}(L_{2i}-2) &  \text{if } 1\leq i<j \leq \frac{N-1}{2}, \\
2(F_{N}-F_{N-2j}) +F_{N-2i}(L_{2j}-2) &  \text{if } 1\leq j\leq i \leq \frac{N-1}{2}, \\
2(F_{N}-F_{N-2j}) &  \text{if } i=\frac{N+1}{2},\ 1\leq j\leq \frac{N-1}{2},\\
F_{N}-F_{N-2i} &  \text{if } j=\frac{N+1}{2},\ 1\leq i\leq \frac{N-1}{2},\\
F_{N} &  \text{if } i=j=\frac{N+1}{2}.
\end{cases}
\]
We show that $H_{N+1}K_{N+1}=I_{\frac{N+1}{2}}$.
First, we demonstrate that the diagonal elements are equal to $1$. 

\begin{itemize}
\item
We calculate that 
$H_{N+1}K_{N+1}(1, 1)=1$. 
\begin{align*}
& \sum_{k=1}^{\frac{N+1}{2}}H_{N+1}(1,k)K_{N+1}(k,1) \\
= & 3K_{N+1}(1,1)-K_{N+1}(2,1)-K_{N+1}\left(\frac{N+1}{2},1\right) \\
= & \frac{1}{F_{N-1}+F_{N+1}} \times \\
 & ( 6( F_N -F_{N-2}) +3F_{N-2} (L_{2}-2) -2( F_N -F_{N-2}) -F_{N-4} (L_{2}-2) -2( F_N -F_{N-2}))\\
= & \frac{1}{F_{N-1}+F_{N+1}}( 2( F_N -F_{N-2}) +(L_2-2)(3F_{N-2} -F_{N-4})) \\
= & \frac{1}{F_{N-1}+F_{N+1}}( 2F_{N-1} +2F_{N-2} + F_{N-3})\\
= & \frac{1}{F_{N-1}+F_{N+1}}( 3F_{N-1} + F_{N-2})\\
= & \frac{1}{F_{N-1}+F_{N+1}}( 2F_{N-1} + F_N)\\
= & \frac{1}{F_{N-1}+F_{N+1}}( F_{N-1} + F_{N+1}) \\
= & 1.
\end{align*}

\item
We calculate that 
$H_{N+1}K_{N+1}\left(\frac{N-1}{2},\frac{N-1}{2}\right)=1$. 
\begin{align*}
& \sum_{k=1}^{\frac{N+1}{2}}H_{N+1}\left(\frac{N-1}{2},k\right)K_{N+1}\left(k,\frac{N-1}{2}\right) \\
= & -K_{N+1}\left(\frac{N-3}{2},\frac{N-1}{2}\right)+2K_{N+1}\left(\frac{N-1}{2},\frac{N-1}{2}\right)-K_{N+1}\left(\frac{N+1}{2},\frac{N-1}{2}\right) \\
 & \frac{1}{F_{N-1}+F_{N+1}}( -2(F_N-F_3)-F_1(L_{N-3}-2) +4(F_N-F_1) +2F_1(L_{N-1}-2) -2(F_N -F_1) )\\
= & \frac{1}{F_{N-1}+F_{N+1}}( 2F_{3} -F_{1} (L_{N-3}-2) -4F_{1} +2F_{1} (L_{N-1}-2) +2F_{1})\\
= & \frac{1}{F_{N-1}+F_{N+1}}( 2(F_{3} -2F_{1}) +F_{1}( 2L_{N-1} -L_{N-3}))\\
= & \frac{1}{F_{N-1}+F_{N+1}}( L_{N-1} +L_{N-2}) \\
= & \frac{L_N}{F_{N-1}+F_{N+1}} \\
= & \frac{F_{N-1}+F_{N+1}}{F_{N-1}+F_{N+1}} \\
= & 1.
\end{align*}

\item
We calculate that 
$H_{N+1}K_{N+1}(\frac{N+1}{2}, \frac{N+1}{2})=1$. 
\begin{align*}
  & \sum_{k=1}^{\frac{N+1}{2}}H_{N+1}\left(\frac{N+1}{2},k\right)K_{N+1}\left(k,\frac{N+1}{2}\right) \\
= & -2\sum_{k=1}^{\frac{N-1}{2}}K_{N+1}\left(k,\frac{N+1}{2}\right)+NK_{N+1}\left(\frac{N+1}{2},\frac{N+1}{2}\right) \\
= & \frac{1}{F_{N-1}+F_{N+1}}\left(-\sum_{k=1}^{\frac{N-1}{2}} 2(F_N-F_{N-2k}) + NF_N\right) \\
= & \frac{1}{F_{N-1}+F_{N+1}}\left(-(N-1)F_N+2\sum_{k=1}^{\frac{N-1}{2}}F_{N-2k} + NF_N\right) \\
= & \frac{1}{F_{N-1}+F_{N+1}}\left( F_N +2\sum_{k=0}^{\frac{N-3}{2}}F_{2k+1} \right)\\
= & \frac{1}{F_{N-1}+F_{N+1}}(F_N+2F_{N-1} )\\
= & \frac{1}{F_{N-1}+F_{N+1}}(F_{N-1}+F_{N+1} )\\
= & 1.
\end{align*}

\item
We calculate that 
$H_{N+1}K_{N+1}(\ell, \ell)=1$, where $1<\ell<\frac{N-1}{2}$. 
\begin{align*}
& \sum_{k=1}^{\frac{N+1}{2}}H_{N+1}(\ell,k)K_{N+1}(k,\ell) \\
= & -K_{N+1}(\ell-1,\ell) +3K_{N+1}(\ell,\ell) -K_{N+1}(\ell+1,\ell) -K_{N+1}\left(\frac{N+1}{2},\ell\right) \\
= & \frac{1}{F_{N-1}+F_{N+1}}(-2(F_N-F_{N+2-2\ell}) -F_{N-2\ell}(L_{2\ell-2}-2) +6(F_N-F_{N-2\ell}) +3F_{N-2\ell}(L_{2\ell}-2) \\
 & -2(F_N-F_{N-2\ell}) -F_{N-2-2\ell}(L_{2\ell}-2) -2( F_N -F_{N-2\ell})) \\
%= & \frac{1}{F_{N-1}+F_{N+1}}(-2F_N +2F_{N+2-2\ell} -F_{N-2\ell}(L_{2\ell-2}-2) +6F_N -6F_{N-2\ell} +3F_{N-2\ell}(L_{2\ell}-2) \\
% & -2F_N +2F_{N-2\ell} -F_{N-2-2\ell}(L_{2\ell}-2) -2F_N +2F_{N-2\ell}) \\
= & \frac{1}{F_{N-1}+F_{N+1}}( 2(F_{N+2-2\ell} -F_{N-2\ell}) +(L_{2\ell}-2) (3F_{N-2\ell} -F_{N-2-2\ell}) -F_{N-2\ell}(L_{2\ell-2}-2)) \\
%= & \frac{1}{F_{N-1}+F_{N+1}}( 2F_{N+1-2\ell} +F_{N+2-2\ell}M_{2\ell} -F_{N-2\ell}M_{2\ell-2}) \\
= & \frac{1}{F_{N-1}+F_{N+1}}(F_{N+2-2\ell}L_{2\ell} -F_{N+1-2\ell}L_{2\ell-2}) \\
= & \frac{1}{F_{N-1}+F_{N+1}}(F_{N+2}-F_{4\ell-N-2} -F_{N-2}+F_{4\ell-N-2}) \quad \text{[by \eqref{formula5}]}\\
= & \frac{1}{F_{N-1}+F_{N+1}}(F_{N-1}+F_{N+1}) \\
= & 1.
\end{align*}
%The third equality is derived from $F_{m+p}+(-1)^{p+1}F_{m-p}=F_pL_m$. 
\end{itemize}

Thus, we have shown that the diagonal elements are equal to $1$.

Next, we show that the other elements are equal to $0$. 

\begin{itemize}
\item
We calculate that $H_{N+1}K_{N+1}(1, \ell)=0$, where $2 \leq \ell \leq \frac{N-1}{2}$. 
\begin{align*}
& \sum_{k=1}^{\frac{N+1}{2}}H_{N+1}(1,k)K_{N+1}(k,\ell) \\
= & 3K_{N+1}(1,\ell) -K_{N+1}(2,\ell) -K_{N+1}\left(\frac{N+1}{2},\ell\right) \\
= & \frac{1}{F_{N-1}+F_{N+1}}(6(F_N-F_{N-2})+3F_{N-2\ell}(L_2-2) \\
 & -2(F_N-F_{N-4})-F_{N-2\ell}(L_4-2) -2(F_N-F_{N-2\ell})) \\
= & \frac{1}{F_{N-1}+F_{N+1}}(2(F_N -3F_{N-2} +F_{N-4}) \\
 & +3F_{N-2\ell}(L_2-2) -F_{N-2\ell}(L_4-2) +2F_{N-2\ell}) \\
= & \frac{1}{F_{N-1}+F_{N+1}}(2(F_{N-1} -2F_{N-2} +F_{N-4}) +3F_{N-2\ell} -5F_{N-2\ell} +2F_{N-2\ell}) \\
= & \frac{1}{F_{N-1}+F_{N+1}}2(F_{N-3} -F_{N-2} +F_{N-4}) \\
= & 0.
\end{align*}

\item
We calculate that $H_{N+1}K_{N+1}(1,\frac{N+1}{2})=0$. 
\begin{align*}
  & \sum_{k=1}^{\frac{N+1}{2}}H_{N+1}(1,k)K_{N+1}\left(k,\frac{N+1}{2}\right) \\
= & 3K_{N+1}\left(1,\frac{N+1}{2}\right)-K_{N+1}\left(2,\frac{N+1}{2}\right) -K_{N+1}\left(\frac{N+1}{2},\frac{N+1}{2}\right) \\
= & \frac{1}{F_{N-1}+F_{N+1}}(3(F_N-F_{N-2}) -F_N +F_{N-4} -F_N) \\
= & \frac{1}{F_{N-1}+F_{N+1}}(F_N -3F_{N-2} +F_{N-4}) \\
= & \frac{1}{F_{N-1}+F_{N+1}}(F_{N-1} -2F_{N-2} +F_{N-4}) \\
= & \frac{1}{F_{N-1}+F_{N+1}}(F_{N-2} -F_{N-2} +F_{N-4}) \\
= & 0.
\end{align*}

\item
We calculate that $H_{N+1}K_{N+1}(\frac{N-1}{2}, \frac{N+1}{2})=0$. 
\begin{align*}
& \sum_{k=1}^{\frac{N+1}{2}}H_{N+1}\left(\frac{N-1}{2},k\right)K_{N+1}\left
(k,\frac{N+1}{2}\right) \\
= & -K_{N+1}\left(\frac{N-3}{2},\frac{N+1}{2}\right) +2K_{N+1}\left(\frac{N-1}{2},\frac{N+1}{2}\right) -K_{N+1}\left(\frac{N+1}{2},\frac{N+1}{2}\right) \\
= & \frac{1}{F_{N-1}+F_{N+1}}(-F_N +F_3 +2F_N -2F_1 -F_N) \\
= & 0.
\end{align*}

\item
We calculate that $H_{N+1}K_{N+1}(\frac{N-1}{2}, \ell)=0$, where $1\leq \ell\leq \frac{N-3}{2}$. 
\begin{align*}
& \sum_{k=1}^{\frac{N+1}{2}}H_{N+1}\left(\frac{N-1}{2},k\right)K_{N+1}(k,\ell) \\
= & -K_{N+1}\left(\frac{N-3}{2},\ell\right) +2K_{N+1}\left(\frac{N-1}{2},\ell\right) -K_{N+1}\left(\frac{N+1}{2},\ell\right) \\
= & \frac{1}{F_{N-1}+F_{N+1}} \times \\
 & (-2(F_N-F_{N-2\ell}) -F_3(L_{2\ell}-2) +4(F_N-F_{2n+1-2\ell}) +2F_1(L_{2\ell}-2) -2(F_N-F_{N-2\ell})) \\
= & \frac{1}{F_{N-1}+F_{N+1}}(-F_3(L_{2\ell}-2) +2F_1(L_{2\ell}-2)) \\
= & \frac{1}{F_{N-1}+F_{N+1}}(-2(L_{2\ell}-2) +2(L_{2\ell}-2)) \\
= & 0.
\end{align*}

\item
We calculate that $H_{N+1}K_{N+1}(\frac{N+1}{2}, \ell)=0$, where $1\leq \ell \leq \frac{N-1}{2}$. 
\begin{align*}
& \sum_{k=1}^{\frac{N+1}{2}}H_{N+1}\left(\frac{N+1}{2},k\right)K_{N+1}(k,\ell) \\
= & -2\sum_{k=1}^{\ell-1}K_{N+1}(k,\ell) -2\sum_{k=\ell}^{\frac{N-1}{2}}K_{N+1}(k,\ell) +NK_{N+1}\left(\frac{N+1}{2},\ell\right) \\
= & \frac{1}{F_{N-1}+F_{N+1}}\left(-2\sum_{k=1}^{\ell-1}(2(F_N-F_{N-2k})+F_{N-2\ell}(L_{2k}-2)) \right. \\
 & \left.-2\sum_{k=\ell}^{\frac{N-1}{2}}(2(F_N-F_{N-2\ell})+F_{N-2k}(L_{2\ell}-2)) +2N(F_N-F_{N-2\ell})\right) \\
= & \frac{1}{F_{N-1}+F_{N+1}}\left(-4F_N\sum_{k=1}^{\ell-1}1 +4\sum_{k=1}^{\ell-1}F_{N-2k}-2F_{N-2\ell}\sum_{k=1}^{\ell-1}(L_{2k}-2)\right. \\
 & \left.-4F_N\sum_{k=\ell}^{\frac{N-1}{2}}1 +4F_{N-2\ell}\sum_{k=\ell}^{\frac{N-1}{2}}1 -2(L_{2\ell}-2)\sum_{k=\ell}^{\frac{N-1}{2}}F_{N-2k} +2NF_N -2NF_{N-2\ell}\right) \\
= & \frac{1}{F_{N-1}+F_{N+1}}\left(-4F_N\sum_{k=1}^{\frac{N-1}{2}}1 +4\sum_{k=\frac{N-2\ell+1}{2}}^{\frac{N-3}{2}}F_{2k+1} -2F_{N-2\ell}\sum_{k=1}^{\ell-1}L_{2k} +4F_{N-2\ell}\sum_{k=1}^{\ell-1}1 \right. \\
 & \left. +4F_{N-2\ell}\sum_{k=\ell}^{\frac{N-1}{2}}1 -2L_{2\ell}\sum_{k=0}^{\frac{N-2\ell-1}{2}}F_{2k+1} +4\sum_{k=0}^{\frac{N-2\ell-1}{2}}F_{2k+1} +2NF_N -2NF_{N-2\ell}\right) \\
= & \frac{1}{F_{N-1}+F_{N+1}}(-2(N-1)F_N +4F_{N-1} -2F_{N-2\ell}(L_{2\ell-1}-1) \\
 & +2(N-1)F_{N-2\ell} -2L_{2\ell}F_{N+1-2\ell} +2NF_N -2NF_{N-2\ell}) \\
= & \frac{1}{F_{N-1}+F_{N+1}}(2(2F_{N-1} +F_N) -2F_{N-2\ell}L_{2\ell-1} -2L_{2\ell}F_{N+1-2\ell}) \\
= & \frac{1}{F_{N-1}+F_{N+1}}(2(F_{N-1} +F_{N+1}) -2(F_{N-1}-F_{N+1-4\ell}) -2(F_{N+1}+F_{N+1-4\ell})) &&\text{[by \eqref{formula5}]}\\
= & 0. \\
\end{align*}

Finally, we show that all non-diagonal elements equal $0$, thereby completing the proof of Proposition~$1$.
\item
For $1<j<i<\frac{N-1}{2}$, we have 
\begin{align*}
& \sum_{k=1}^{\frac{N+1}{2}}H_{N+1}(i,k)K_{N+1}(k,j) \\
= & -K_{N+1}(i-1,j) +3K_{N+1}(i,j) -K_{N+1}(i+1,j) -K_{N+1}\left(\frac{N+1}{2},j\right) \\
= & \frac{1}{F_{N-1}+F_{N+1}}(-2(F_N-F_{N-2j}) -F_{N+2-2i}(L_{2j}-2) +6(F_N-F_{N-2j}) +3F_{N-2i}(L_{2j}-2) \\
 & -2(F_N-F_{N-2j}) -F_{N-2-2i}(L_{2j}-2) -2(F_N-F_{N-2j})) \\
= & \frac{1}{F_{N-1}+F_{N+1}}((L_{2j}-2)(-F_{N+2-2i} +3F_{N-2i} -F_{N-2-2i})) \\
= & \frac{1}{F_{N-1}+F_{N+1}}((L_{2j}-2)(-F_{N-1-2i} +F_{N-2i} -F_{N-2-2i})) \\
= & 0.
\end{align*}

\item
For $1<i< j<\frac{N-1}{2}$, we have
\begin{align*}
& \sum_{k=1}^{\frac{N+1}{2}}H_{N+1}(i,k)K_{N+1}(k,j) \\
= & -K_{N+1}(i-1,j) +3K_{N+1}(i,j) -K_{N+1}(i+1,j) -K_{N+1}\left(\frac{N+1}{2},j\right) \\
= & \frac{1}{F_{N-1}+F_{N+1}}(-2(F_N-F_{N+2-2i}) -F_{N-2j}(L_{2i-2}-2) +6(F_N-F_{N-2i}) +3F_{N-2j}(L_{2i}-2) \\
 & -2(F_N-F_{N-2-2i}) -F_{N-2j}(L_{2i+2}-2) -2(F_N-F_{N-2j})) \\
= & \frac{1}{F_{N-1}+F_{N+1}} \times \\
 & (2(F_{N+2-2i} -2F_{N-2i} +F_{N-2-2i}) -F_{N-2j}(L_{2i-2}-2 -3(L_{2i}-2) +L_{2i+2}-2)) \\
= & \frac{1}{F_{N-1}+F_{N+1}}(2F_{N-2i} -F_{N-2j}(L_{2i-2} -3L_{2i} +L_{2i+2} +2)) \\
= & \frac{1}{F_{N-1}+F_{N+1}}(-F_{N-2j}(L_{2i-2} -L_{2i} +L_{2i-1})) \\
= & 0.
\end{align*}
 
\end{itemize}
Thus, we have shown that all non-diagonal elements are $0$. Therefore, $K_{N+1}=H_{N+1}^{-1}$ is proven.
\end{proof}

By combining the above proposition and the matrix equation, we obtain the exact formula for the average hitting times of simple random walks on $W_{N+1}$ when the number of vertices in the cycle part is odd as follows.

For $1\leq \ell \leq \lfloor\frac{N}{2}\rfloor$, we have
\begin{align*}
h(W_{N+1};0,\ell)
%&= h'_\ell\\
& = 3\sum_{k=1}^{\frac{N-1}{2}}H_{N+1}^{-1}(\ell,k)+NH_{N+1}^{-1}\left(\ell,\frac{N+1}{2}\right) \\
& =\frac{1}{F_{N-1}+F_{N+1}} \left(3\sum_{k=1}^{\ell}(2(F_N-F_{N-2k})+F_{N-2\ell}(L_{2k}-2))\right. \\
& \left. +3\sum_{k=\ell+1}^{\frac{N-1}{2}}(2(F_N-F_{N-2\ell})+F_{N-2k}(L_{2\ell}-2))+N(F_N-F_{N-2\ell})\right) \\
& =\frac{1}{F_{N-1}+F_{N+1}}\left(\sum_{k=1}^{\frac{N-1}{2}}6F_N-\sum_{k=1}^{\frac{N-1}{2}}6F_{N-2k}+\sum_{k=1}^{\ell}3F_{N-2\ell}(L_{2k}-2) \right. \\
& \left. -\sum_{k=1}^{\frac{N-2\ell-1}{2}}6F_{N-2\ell}+3L_{2\ell}\sum_{k=1}^{\frac{N-2\ell-1}{2}}F_{2k-1}+N(F_N-F_{N-2\ell})\right) \\
& =\frac{1}{F_{N-1}+F_{N+1}}\left(3(N-1)F_N-\sum_{k=1}^{\frac{N-1}{2}}6F_{2k-1}+3F_{N-2\ell}\sum_{k=1}^{\ell}L_{2k}-6\ell F_{N-2\ell} \right. \\
& \left. -3(N-2\ell-1)F_{N-2\ell}+3L_{2\ell}\sum_{k=1}^{\frac{N-2\ell-1}{2}}F_{2k-1}+N(F_N-F_{N-2\ell})\right) \\
& =\frac{1}{F_{N-1}+F_{N+1}} (3(N-1)F_N-6F_{N-1}+3F_{N-2\ell}(L_{2\ell+1}-1)-6\ell F_{N-2\ell} \\
& -3(N-2\ell-1)F_{N-2\ell}+3L_{2\ell}F_{N-2\ell-1}+N(F_N-F_{N-2\ell})) \quad\text{[by \eqref{Fib8} and \eqref{Luc9}]}\\
& =\frac{1}{F_{N-1}+F_{N+1}} ((4N-3)F_N-6F_{N-1}-4NF_{N-2\ell} \\
& +3(F_{N-2\ell}L_{2\ell+1}+L_{2\ell}F_{N-2\ell-1})) \\
& =\frac{1}{F_{N-1}+F_{N+1}} ((4N-3)F_N-6F_{N-1}-4NF_{N-2\ell}+3L_N) \quad \text{[by \eqref{formula6}]}\\
& =\frac{1}{F_{N-1}+F_{N+1}} ((4N-3)F_N-6F_{N-1}-4NF_{N-2\ell}+3(F_{N-1}+F_{N+1}))\\
& =\frac{4N(F_N-F_{N-2\ell})}{F_{N-1}+F_{N+1}}.
\end{align*}

By \eqref{Fib12}, we have $F_{N-2\ell}=F_{2\ell-N}$.
Therefore, by combining this with $h(W_{N+1};0,\ell)=h(W_{N+1};0,N-\ell)$, we obtain 
\[
h(W_{N+1};0,\ell)=\frac{4N(F_N-F_{N-2\ell})}{F_{N-1}+F_{N+1}} \quad \text{for }0\leq \ell\leq N-1.
\]

Also, we have
\begin{align*}
h_{N,0}
= & 3\sum_{k=1}^{\frac{N-1}{2}}H_{N+1}^{-1}\left(\frac{N+1}{2},k\right)+NH_{N+1}^{-1}\left(\frac{N+1}{2},\frac{N+1}{2}\right) \\
= & \frac{1}{F_{N-1}+F_{N+1}}\left(6\sum_{k=1}^{\frac{N-1}{2}}(F_N-F_{N-2k}) +NF_N\right) \\
= & \frac{1}{F_{N-1}+F_{N+1}}\left(6F_N\sum_{k=1}^{\frac{N-1}{2}}1 -6\sum_{k=0}^{\frac{N-3}{2}}F_{2k+1} +NF_N\right) \\
= & \frac{1}{F_{N-1}+F_{N+1}} (3(N-1)F_N -6F_{N-1} +NF_N) \\
= & \frac{1}{F_{N-1}+F_{N+1}} ((4N-3)F_N -6F_{N-1}) \\
= & \frac{1}{F_{N-1}+F_{N+1}} ((4N-3)F_{N+1} -(4N+3)F_{N-1}).
\end{align*}

%%%%%%%%%%%%%%%%%%%%%%%%%%%%%%%%%%%
%%%%%%%%%%%%%%%%%%%%%%%%%%%%%%%%%%%
%偶数頂点の場合
%%%%%%%%%%%%%%%%%%%%%%%%%%%%%%%%%%%
%%%%%%%%%%%%%%%%%%%%%%%%%%%%%%%%%%%
\subsection{Even \label{subsec042}}

Next, we consider the case where the wheel graph $W_{N+1}$ has an even number of vertices in its cycle part. 
In this case, the coefficient matrix $H_{N+1}$ is defined similarly to the case of an odd number of vertices, as follows: 

\begin{center}
$H_{N+1}(i,j)=
\begin{cases}
3 & \text{if } i=j\ ,\ i\neq\frac{N+2}{2},\\
N & \text{if } i=j=\frac{N+2}{2},\\
-1 & \text{if } |i-j|=1,( i,j) \neq ( \frac{N}{2},\frac{N-2}{2}),\\
 & j=\frac{N+2}{2},( i,j) \neq ( \frac{N+2}{2},\frac{N+2}{2}),\\
-2 & \text{if } i=\frac{N+2}{2},( i,j) \neq ( \frac{N+2}{2},\frac{N}{2}),\\
 & ( i,j) =( \frac{N}{2},\frac{N-2}{2}),\\
0 & \mbox{otherwise.}
\end{cases}$
\end{center}

It is important to note that similar to the case with an odd number of vertices, the coefficient matrix is again a square matrix of size $\frac{N+2}{2} \times \frac{N+2}{2}$.
The last row represents the average hitting times from the center to the peripheral vertices. 

For example, the coefficient matrix for $W_{8+1}$ can be written as follows: 
\begin{align*}
H_{8+1}
&=
\begin{bmatrix}
3 & -1 & 0 & 0 & -1\\
-1 & 3 & -1 & 0 & -1\\
0 & -1 & 3 & -1 & -1\\
0 & 0 & -2 & 3 & -1\\
-2 & -2 & -2 & -1 & 8
\end{bmatrix}.
\end{align*}

%%%%%  命題  %%%%%
\begin{prop} \label{prop:2}
The inverse of the coefficient matrix $H_{N+1}$ for a wheel graph with an even number of vertices in the cycle part can be expressed as follows: 

\[
H_{N+1}^{-1}(i,j)
=\frac{1}{L_{N-1}+L_{N+1}}
\begin{cases}
2(L_{N} -L_{N-2i}) +L_{N-2j}(L_{2i}-2) & \text{if } 1\leq i\leq j \leq \frac{N-2}{2},\\
2(L_{N} -L_{N-2j}) +L_{N-2i}(L_{2j}-2) & \text{if } 1\leq j< i \leq \frac{N}{2},\\
L_{N} -L_{N-2i} +L_{2i}-2 & \text{if } j=\frac{N}{2}, 1\leq i\leq \frac{N}{2}, \\
%2(L_{N}-2) & \text{if } i=j=\frac{N}{2} \\
2(L_{N} -L_{N-2j}) & \text{if } i=\frac{N+2}{2},\ 1\leq j\leq \frac{N-2}{2},\\
L_{N}-2 & \text{if } (i, j)=(\frac{N+2}{2}, \frac{N}{2}), \\
L_{N} -L_{N-2i} & \text{if } j=\frac{N+2}{2},\ 1\leq i\leq \frac{N}{2},\\
L_{N} & \text{if } i=j=\frac{N+2}{2}.
\end{cases}
\]
\end{prop}
%%%%%%%%%%%%%%%

\begin{proof} 
The matrix $K_{N+1}$ is defined as follows: 

\[
K_{N+1}(i,j)
:=\frac{1}{L_{N-1}+L_{N+1}}
\begin{cases}
2(L_{N} -L_{N-2i}) +L_{N-2j}(L_{2i}-2) & \text{if } 1\leq i\leq j\leq\frac{N-2}{2},\\
2(L_{N} -L_{N-2j}) +L_{N-2i}(L_{2j}-2) & \text{if } 1\leq j< i\leq \frac{N}{2},\\
L_{N} -L_{N-2i} +L_{2i}-2 & \text{if } j=\frac{N}{2}, 1\leq i\leq \frac{N}{2}, \\
%2(L_{N}-2) & \text{if } i=j=\frac{N}{2} \\
2(L_{N} -L_{N-2j}) & \text{if } i=\frac{N+2}{2},\ 1\leq j\leq \frac{N-2}{2},\\
L_{N}-2 & \text{if } (i, j)=(\frac{N+2}{2}, \frac{N}{2}), \\
L_{N} -L_{N-2i} & \text{if } j=\frac{N+2}{2},\ 1\leq i\leq \frac{N}{2},\\
L_{N} & \text{if } i=j=\frac{N+2}{2}.
\end{cases}
\]

We will show that 
$H_{N+1}K_{N+1}=I_{\frac{N+2}{2}}$. 
First, we show that the diagonal elements are equal to $1$. 

\begin{itemize}
\item
We calculate that 
$H_{N+1}K_{N+1}(1,1)=1$. 
\begin{align*}
&\sum_{k=1}^{\frac{N+2}{2}}H_{N+1}(1,k)K_{N+1}(k,1) \\
= &
3K_{N+1}(1,1)-K_{N+1}(2,1)-K_{N+1}\left(\frac{N+2}{2},1\right) \\
= & \frac{1}{L_{N-1}+L_{N+1}} \times \\
 & (3(2(L_N-L_{N-2})+L_{N-2}(L_2-2))-(2(L_N-L_{N-2})+L_{N-4}(L_2-2))-2(L_N-L_{N-2})) \\
= & \frac{1}{L_{N-1}+L_{N+1}}(2(L_N-L_{N-2})+3L_{N-2}(L_2-2)-L_{N-4}(L_2-2)) \\
= & \frac{1}{L_{N-1}+L_{N+1}}(2L_N+L_{N-2}-L_{N-4}) \\
= & \frac{1}{L_{N-1}+L_{N+1}}(L_{N-1}+L_{N+1}) \\
= & 1.
\end{align*}

\item
We calculate that 
$H_{N+1}K_{N+1}\left(\frac{N}{2},\frac{N}{2}\right)=1$. 
\begin{align*}
&\sum_{k=1}^{\frac{N+2}{2}}H_{N+1}\left(\frac{N}{2},k\right)K_{N+1}\left(k,\frac{N}{2}\right) \\
= & -2K_{N+1}\left(\frac{N-2}{2},\frac{N}{2}\right)+3K_{N+1}\left(\frac{N}{2},\frac{N}{2}\right)-K_{N+1}\left(\frac{N+2}{2},\frac{N}{2}\right) \\
= & \frac{1}{L_{N-1}+L_{N+1}}(-2(L_N-L_2+L_{N-2}-2)+6(L_N-2)-(L_N-2)) \\
= & \frac{1}{L_{N-1}+L_{N+1}}(-2L_{N-2}+3L_N) \\
= & \frac{1}{L_{N-1}+L_{N+1}}(L_{N-1}+L_{N+1}) \\
= & 1.
\end{align*}

\item
We calculate that 
$H_{N+1}K_{N+1}(\frac{N+2}{2},\frac{N+2}{2})=1$. 
\begin{align*}
&\sum_{k=1}^{\frac{N+2}{2}}H_{N+1}\left(\frac{N+2}{2},k\right)K_{N+1}\left(k,\frac{N+2}{2}\right) \\
= & -2\sum_{k=1}^{\frac{N-2}{2}}K_{N+1}\left(k,\frac{N+2}{2}\right)-K_{N+1}\left(\frac{N}{2},\frac{N+2}{2}\right)+NK_{N+1}\left(\frac{N+2}{2},\frac{N+2}{2}\right) \\
= & \frac{1}{L_{N-1}+L_{N+1}}\left(-2\sum_{k=1}^{\frac{N-2}{2}}(L_N-L_{N-2k})-(L_N-L_0)+NL_N\right) \\
= & \frac{1}{L_{N-1}+L_{N+1}}\left(-(N-2)L_N+2\sum_{k=1}^{\frac{N-2}{2}}L_{2k}-(L_N-2)+NL_N\right) \\
= & \frac{1}{L_{N-1}+L_{N+1}}(-(N-2)L_N+2(L_{N-1}-1)-(L_N-2)+NL_N) &&\text{[by \eqref{Luc9}]}\\
= & \frac{1}{L_{N-1}+L_{N+1}}(2L_{N-1}+L_N) \\
= & 1.
\end{align*}

\item
We calculate that 
$H_{N+1}K_{N+1}(\ell,\ell)=1$, where $1< \ell<\frac{N}{2}$.
\begin{align*}
&\sum_{k=1}^{\frac{N+2}{2}}H_{N+1}(\ell,k)K_{N+1}(k,\ell) \\
= &
-K_{N+1}(\ell-1,\ell)+3K_{N+1}(\ell,\ell)-K_{N+1}(\ell+1,\ell)-K_{N+1}\left(\frac{N+2}{2},\ell\right) \\
= &\frac{1}{L_{N-1}+L_{N+1}}(-2(L_N-L_{N-2\ell+2})-L_{N-2\ell}(L_{2\ell-2}-2)+6(L_N-L_{N-2\ell})+3L_{N-2\ell}(L_{2\ell}-2) \\
 & -2(L_N-L_{N-2\ell})-L_{N-2\ell-2}(L_{2\ell}-2)-2(L_N-L_{N-2\ell})) \\
= &\frac{1}{L_{N-1}+L_{N+1}}(2L_{N-2\ell+2}-6L_{N-2\ell}-L_{N-2\ell}(L_{2\ell-2}-3L_{2\ell})-L_{N-2\ell-2}(L_{2\ell}-2)) \\
= &\frac{1}{L_{N-1}+L_{N+1}}(2L_{N-2\ell+2}-6L_{N-2\ell}+2L_{N-2\ell-2}-L_{N-2\ell}(L_{2\ell-2}-3L_{2\ell})-L_{N-2\ell-2}L_{2\ell}) \\
= &\frac{1}{L_{N-1}+L_{N+1}}(L_{N-2\ell}L_{2\ell+2}-L_{N-2\ell-2}L_{2\ell}) \quad \text{[by \eqref{Fib2}]}\\
= &\frac{1}{L_{N-1}+L_{N+1}}(L_{N-2\ell}L_{2\ell+1}+L_{N-2\ell-1}L_{2\ell})\\
= &\frac{1}{L_{N-1}+L_{N+1}}(L_{N-2\ell}F_{2\ell+2}+L_{N-2\ell-1}F_{2\ell+1}+L_{N-2\ell}F_{2\ell}+L_{N-2\ell-1}F_{2\ell-1})\\
= &\frac{1}{L_{N-1}+L_{N+1}}(L_{N-1}+L_{N+1}) \quad \text{[by \eqref{formula6}]}\\
= & 1.
\end{align*}
\end{itemize}

Therefore, we have shown that all diagonal elements are equal to $1$.
Finally, we show that the other elements of the matrix are equal to $0$.

\begin{itemize}
\item
We calculate that 
$H_{N+1}K_{N+1}(1,\ell)=0$, where $1< \ell<\frac{N}{2}$.
\begin{align*}
\sum_{k=1}^{\frac{N+2}{2}}H_{N+1}(1,k)K_{N+1}(k,\ell)
= & 3K_{N+1}(1,\ell)-K_{N+1}(2,\ell)-K_{N+1}\left(\frac{N+2}{2},\ell\right) \\
= & \frac{1}{L_{N-1}+L_{N+1}}(3(2(L_N-L_{N-2})+L_{N-2\ell}(L_2-2)) \\
 & -(2(L_N-L_{N-4})+L_{N-2\ell}(L_4-2))-2(L_N-L_{N-2\ell})) \\
= & \frac{1}{L_{N-1}+L_{N+1}}(2L_N-6L_{N-2}+2L_{N-4}) \\
= & 0. &&\text{[by \eqref{Fib1}]}
\end{align*}

\item
We calculate that 
$H_{N+1}K_{N+1}(1,\frac{N}{2})=0$.
\begin{align*}
&\sum_{k=1}^{\frac{N+2}{2}}H_{N+1}(1,k)K_{N+1}\left(k,\frac{N}{2}\right) \\
= & 3K_{N+1}\left(1,\frac{N}{2}\right)-K_{N+1}\left(2,\frac{N}{2}\right)-K_{N+1}\left(\frac{N+2}{2},\frac{N}{2}\right) \\
= & \frac{1}{L_{N-1}+L_{N+1}}(3(L_N-L_{N-2}+L_2-2)-(L_N-L_{N-4}+L_4-2)-(L_N-2)) \\
= & \frac{1}{L_{N-1}+L_{N+1}}(L_N-3L_{N-2}+L_{N-4}) \\
= & 0. &&\text{[by \eqref{Fib2}]}
\end{align*}

\item
We calculate that 
$H_{N+1}K_{N+1}(1,\frac{N+2}{2})=0$.
\begin{align*}
&\sum_{k=1}^{\frac{N+2}{2}}H_{N+1}(1,k)K_{N+1}\left(k,\frac{N+2}{2}\right) \\
= & 3K_{N+1}\left(1,\frac{N+2}{2}\right)-K_{N+1}\left(2,\frac{N+2}{2}\right)-K_{N+1}\left(\frac{N+2}{2},\frac{N+2}{2}\right) \\
= & \frac{1}{L_{N-1}+L_{N+1}}(3(L_N-L_{N-2})-(L_N-L_{N-4})-L_N) \\
= & \frac{1}{L_{N-1}+L_{N+1}}(L_N-3L_{N-2}+L_{N-4}) \\
= & 0. &&\text{[by \eqref{Fib2}]}
\end{align*}

\item
We calculate that 
$H_{N+1}K_{N+1}(\frac{N}{2},\ell)=0$, where $1\leq \ell<\frac{N}{2}$.
\begin{align*}
&\sum_{k=1}^{\frac{N+2}{2}}H_{N+1}\left(\frac{N}{2},k\right)K_{N+1}(k,\ell) \\
= & -2K_{N+1}\left(\frac{N-2}{2},\ell\right)+3K_{N+1}\left(\frac{N}{2},\ell\right)-K_{N+1}\left(\frac{N+2}{2},\ell\right) \\
= & \frac{1}{L_{N-1}+L_{N+1}} \times \\
 & (-2(2(L_N-L_{N-2\ell})+L_2(L_{2\ell}-2))+3(2(L_N-L_{N-2\ell})+L_0(L_{2\ell}-2))-2(L_N-L_{N-2\ell})) \\
= & \frac{1}{L_{N-1}+L_{N+1}} \times \\
 & (-4(L_N-L_{N-2\ell})-6(L_{2\ell}-2)+6(L_N-L_{N-2\ell})+6(L_{2\ell}-2)-2(L_N-L_{N-2\ell})) \\
= & 0.
\end{align*}

\item
We calculate that  
$H_{N+1}K_{N+1}(\frac{N}{2},\frac{N+2}{2})=0$.
\begin{align*}
&\sum_{k=1}^{\frac{N+2}{2}}H_{N+1}\left(\frac{N}{2},k\right)K_{N+1}\left(k,\frac{N+2}{2}\right) \\
= & -2K_{N+1}\left(\frac{N-2}{2},\frac{N+2}{2}\right)+3K_{N+1}\left(\frac{N}{2},\frac{N+2}{2}\right)-K_{N+1}\left(\frac{N+2}{2},\frac{N+2}{2}\right) \\
= & \frac{1}{L_{N-1}+L_{N+1}}(-2(L_N-L_2)+3(L_N-L_0)-L_N) \\
= & \frac{1}{L_{N-1}+L_{N+1}}(-2(L_N-3)+3(L_N-2)-L_N) \\
= & 0.
\end{align*}

\item
We calculate that  
$H_{N+1}K_{N+1}\left(\frac{N+2}{2},\ell\right)=0$, where $1\leq \ell<\frac{N}{2}$.
\begin{align*}
&\sum_{k=1}^{\frac{N+2}{2}}H_{N+1}\left(\frac{N+2}{2},k\right)K_{N+1}(k,\ell) \\
= & -2\sum_{k=1}^{\frac{N-2}{2}}K_{N+1}(k,\ell)-K_{N+1}\left(\frac{N}{2},\ell\right)+NK_{N+1}\left(\frac{N+2}{2},\ell\right) \\
= & \frac{1}{L_{N-1}+L_{N+1}} \times \\
 & \left(-2\left(\sum_{k=1}^{\ell-1}(2(L_N-L_{N-2k})+L_{N-2\ell}(L_{2k}-2))+\sum_{k=\ell}^{\frac{N-2}{2}}(2(L_N-L_{N-2\ell})+L_{N-2k}(L_{2\ell}-2))\right)\right. \\
 & -(2(L_N-L_{N-2\ell})+L_0(L_{2\ell}-2))+2N(L_N-L_{N-2\ell}) \Bigg) \\
= & \frac{1}{L_{N-1}+L_{N+1}}\left(-4L_N\sum_{k=1}^{\ell-1}1+4\sum_{k=1}^{\ell-1}L_{N-2k}-2L_{N-2\ell}\sum_{k=1}^{\ell-1}L_{2k}+4L_{N-2\ell}\sum_{k=1}^{\ell-1}1-4L_N\sum_{k=\ell}^{\frac{N-2}{2}}1 \right.\\
 & \left.+4L_{N-2\ell}\sum_{k=\ell}^{\frac{N-2}{2}}1
-2(L_{2\ell}-2)\sum_{k=\ell}^{\frac{N-2}{2}}L_{N-2k}-2L_N+2L_{N-2\ell}-2L_{2\ell}+4+2N(L_N-L_{N-2\ell})\right) \\
= & \frac{1}{L_{N-1}+L_{N+1}}\left(-4(\ell-1)L_N+4\sum_{k=\frac{N-2\ell+2}{2}}^{\frac{N-2}{2}}L_{2k}-2L_{N-2\ell}\sum_{k=1}^{\ell-1}L_{2k}+4(\ell-1)L_{N-2\ell}-2(N-2\ell)L_N \right.\\
 & \left. +2(N-2\ell)L_{N-2\ell}
-2(L_{2\ell}-2)\sum_{k=1}^{\frac{N-2\ell}{2}}L_{2k}-2L_N+2L_{N-2\ell}-2L_{2\ell}+4+2N(L_N-L_{N-2\ell})\right) \\
= & \frac{2}{L_{N-1}+L_{N+1}} \times \\
 & (2(L_{N-1}-1)-L_{N-2\ell}(L_{2\ell-1}-1)-L_{2\ell}(L_{N-2\ell+1}-1)+L_N-L_{N-2\ell}-L_{2\ell}+2) \quad \text{[by \eqref{Luc9}]} \\\\
= & \frac{2}{L_{N-1}+L_{N+1}}(2L_{N-1}+L_N-L_{N-2\ell}L_{2\ell-1}-L_{2\ell}L_{N-2\ell+1}) \\
= & \frac{2}{L_{N-1}+L_{N+1}}(2L_{N-1}+L_N-L_{N-2\ell}(F_{2\ell-2}+F_{2\ell})-(F_{2\ell-1}+F_{2\ell+1})L_{N-2\ell+1})\\
= & \frac{2}{L_{N-1}+L_{N+1}}(L_{N-1}+L_{N+1}-L_{N-1}-L_{N+1}) \quad \text{[by \eqref{formula6}]} \\
= & 0.
\end{align*}

\item
We calculate that 
$H_{N+1}K_{N+1}(\frac{N+2}{2},\frac{N}{2})=0$.
\begin{align*}
&\sum_{k=1}^{\frac{N+2}{2}}H_{N+1}\left(\frac{N+2}{2},k\right)K_{N+1}\left(k,\frac{N}{2}\right) \\
= & -2\sum_{k=1}^{\frac{N-2}{2}}K_{N+1}\left(k,\frac{N}{2}\right)-K_{N+1}\left(\frac{N}{2},\frac{N}{2}\right)+NK_{N+1}\left(\frac{N+2}{2},\frac{N}{2}\right) \\
= & \frac{1}{L_{N-1}+L_{N+1}}\left(-2\sum_{k=1}^{\frac{N-2}{2}}(L_N-L_{N-2k}+L_{2k}-2)-2(L_N-2)+N(L_N-2)\right)\\
= & \frac{1}{L_{N-1}+L_{N+1}}\left((2\sum_{k=1}^{\frac{N-2}{2}}L_{N-2k}-2\sum_{k=1}^{\frac{N-2}{2}}L_{2k}-2(L_N-2)\sum_{k=1}^{\frac{N-2}{2}}1-2(L_N-2)+N(L_N-2)\right) \\
= & \frac{1}{L_{N-1}+L_{N+1}}\left(2\sum_{k=1}^{\frac{N-2}{2}}L_{2k}-2\sum_{k=1}^{\frac{N-2}{2}}L_{2k}-(N-2)(L_N-2)-2(L_N-2)+N(L_N-2)\right) \\
= & 0.
\end{align*}

Finally, we show that all non-diagonal elements are equal to $0$, thereby completing the proof of Proposition~\ref{prop:2}. 
\item
We calculate that 
$H_{N+1}K_{N+1}(i,j)=0$, where $1<j<i<\frac{N}{2}$.
\begin{align*}
&\sum_{k=1}^{\frac{N+2}{2}}H_{N+1}(i,k)K_{N+1}(k,j) \\
= &-K_{N+1}(i-1,j)+3K_{N+1}(i,j)-K_{N+1}(i+1,j)-K_{N+1}\left(\frac{N+2}{2},j\right) \\
= & \frac{1}{L_{N-1}+L_{N+1}}(-(2(L_N-L_{N-2j})+L_{N-2i+2}(L_{2j}-2))+3(2(L_N-L_{N-2j})+L_{N-2i}(L_{2j}-2)) \\
 & -(2(L_N-L_{N-2j})+L_{N-2i-2}(L_{2j}-2))-2(L_N-L_{N-2j})) \\
= & \frac{1}{L_{N-1}+L_{N+1}}(2(L_{N-2i+2}-3L_{N-2i}+L_{N-2i-2})-L_{2j}(L_{N-2i+2}-3L_{N-2i}+L_{N-2i-2})) \\
= & 0. \quad \text{[by \eqref{Fib1}]}
\end{align*}

\item
We calculate that 
$H_{N+1}K_{N+1}(i,j)=0$, where $1<i<j<\frac{N}{2}$.
\begin{align*}
&\sum_{k=1}^{\frac{N+2}{2}}H_{N+1}(i,k)K_{N+1}(k,j) \\
= & -K_{N+1}(i-1,j)+3K_{N+1}(i,j)-K_{N+1}(i+1,j)-K_{N+1}\left(\frac{N+2}{2},j\right) \\
= & \frac{1}{L_{N-1}+L_{N+1}}(-(2(L_N-L_{N-2i+2})+L_{N-2j}(L_{2i-2}-2))+3(2(L_N-L_{N-2i})+L_{N-2j}(L_{2i}-2)) \\
 & -(2(L_N-L_{N-2i-2})+L_{N-2j}(L_{2i+2}-2))-2(L_N-L_{N-2j})) \\
= & \frac{1}{L_{N-1}+L_{N+1}}(2(L_{N-2i+2}-3L_{N-2i}+L_{N-2i-2})-L_{N-2j}(L_{2i-2}+3L_{2i}-L_{2i+2}) \\
 & +2(L_{N-2j}-3L_{N-2j}+L_{N-2j})) \\
= & 0. \quad \text{[by \eqref{Fib1}]}
\end{align*}

\end{itemize}
Thus, it has been shown that all non-diagonal elements are $0$. Therefore, $K_{N+1} =H_{N+1}^{-1}$ is proven. 
\end{proof}

By combining the above proposition and the matrix equation, we obtain the exact formula for the average hitting times of simple random walks on $W_{N+1}$ when the number of vertices in the cycle part is even as follows.

For $1\leq \ell\leq \lfloor\frac{N}{2}\rfloor$, we have 
\begin{align*}
h(W_{N+1};0,\ell)
= & 3\sum_{k=1}^{\frac{N-2}{2}}H_{N+1}^{-1}(\ell,k)+3H_{N+1}^{-1}\left(\ell,\frac{N}{2}\right)+NH_{N+1}^{-1}\left(\ell,\frac{N+2}{2}\right) \\
= & \frac{1}{L_{N-1}+L_{N+1}}\Bigg(3\sum_{k=1}^{\ell}(2(L_N-L_{N-2k})+L_{N-2\ell}(L_{2k}-2)) \\
 & +3\sum_{k=\ell+1}^{\frac{N-2}{2}}(2(L_N-L_{N-2\ell})+L_{N-2k}(L_{2\ell}-2)) \\
 & +3(L_N-L_{N-2\ell}+L_{2\ell}-2)+N(L_N-L_{N-2\ell})\Bigg) \\
= & \frac{1}{L_{N-1}+L_{N+1}}\left(6(L_N-L_{N-2\ell})\sum_{k=1}^{\frac{N-2}{2}}1-6\sum_{k=1}^{\frac{N-2}{2}}L_{N-2k}+3L_{N-2\ell}\sum_{k=1}^{\ell}L_{2k} \right.\\
 & \left.+3L_{2\ell}\sum_{k=\ell+1}^{\frac{N-2}{2}}L_{N-2k}+3(L_{2\ell}-2)+(N+3)(L_N-L_{N-2\ell})\right) \\
= & \frac{1}{L_{N-1}+L_{N+1}}\left(3(N-2)(L_N-L_{N-2\ell})-6\sum_{k=1}^{\frac{N-2}{2}}L_{2k}+3L_{N-2\ell}(L_{2\ell+1}-1) \right. \\
 & \left. +3L_{2\ell}\sum_{k=1}^{\frac{N-2\ell-2}{2}}L_{2k}+3(L_{2\ell}-2)+(N+3)(L_N-L_{N-2\ell})\right) &&\text{[by \eqref{Luc9}]}\\
= & \frac{1}{L_{N-1}+L_{N+1}}(-6(L_{N-1}-1)+3L_{N-2\ell}(L_{2\ell+1}-1) \\
 & +3L_{2\ell}(L_{N-2\ell-1}-1)+3(L_{2\ell}-2)+(4N-3)(L_N-L_{N-2\ell})) &&\text{[by \eqref{Luc9}]}\\
= & \frac{1}{L_{N-1}+L_{N+1}}(-3(L_{N-1}+L_{N+1})+4N(L_N-L_{N-2\ell}) \\
 & +3(L_{N-2\ell}(F_{2\ell+2}+F_{2\ell})+(F_{2\ell-1}+F_{2\ell+1})L_{N-2\ell-1})) \\
= & \frac{1}{L_{N-1}+L_{N+1}}(-3(L_{N-1}+L_{N+1})+4N(L_N-L_{N-2\ell}) \\
 & +3(F_{2\ell+2}L_{N-2\ell}+F_{2\ell+1}L_{N-2\ell-1}+F_{2\ell}L_{N-2\ell}+F_{2\ell-1}L_{N-2\ell-1})) \\
= & \frac{1}{L_{N-1}+L_{N+1}} \times \\
 & (-3(L_{N-1}+L_{N+1})+4N(L_N-L_{N-2\ell})+3(L_{N+1}+L_{N-1})) &&\text{[by \eqref{formula6}]}\\
= & \frac{4N(L_N-L_{N-2\ell})}{L_{N-1}+L_{N+1}}.
\end{align*}

By \eqref{Luc11}, we have $L_{N-2\ell}=L_{2\ell-N}$.
Therefore, by combining this with $h(W_{N+1};0,\ell)=h(W_{N+1};0,N-\ell)$, we obtain 
\[
h(W_{N+1};0,\ell)=\frac{4N(L_N-L_{N-2\ell})}{L_{N-1}+L_{N+1}} \quad \text{for }0\leq \ell\leq N-1.
\]

Also, we have
\begin{align*}
h(W_{N+1};N,0)
= & 3\sum_{k=1}^{\frac{N-2}{2}}H_{N+1}^{-1}\left(\frac{N+2}{2},k\right)+3H_{N+1}^{-1}\left(\frac{N+2}{2},\frac{N}{2}\right)+NH_{N+1}^{-1}\left(\frac{N+2}{2},\frac{N+2}{2}\right) \\
= & \frac{1}{L_{N-1}+L_{N+1}}\left(3\sum_{k=1}^{\frac{N-2}{2}}(2(L_N-L_{N-2k}))+3(L_N-2)+NL_N\right) \\
= & \frac{1}{L_{N-1}+L_{N+1}}\left(6L_N\sum_{k=1}^{\frac{N-2}{2}}1-6\sum_{k=1}^{\frac{N-2}{2}}L_{N-2k}+3(L_N-2)+NL_N\right) \\
= & \frac{1}{L_{N-1}+L_{N+1}}\left(3(N-2)L_N-6\sum_{k=1}^{\frac{N-2}{2}}L_{2k}+3(L_N-2)+NL_N\right) \\
= & \frac{4NL_N-3L_N-6-6(L_{N-1}-1)}{L_{N-1}+L_{N+1}} \quad \text{[by \eqref{Luc9}]}\\
= & \frac{(4N-3)L_{N+1}-(4N-3)L_{N-1}-6L_{N-1}}{L_{N-1}+L_{N+1}} \\
= & \frac{(4N-3)L_{N+1}-(4N+3)L_{N-1}}{L_{N-1}+L_{N+1}}.
\end{align*}
 
%%%%%%%%%%%%%%%%%%%%%%%%%%%%%%%%%%%%%%%%%%%%%%%%%%%%%%%%%%%%%%%%%%%%%%%%%%%%%%%%%%%%%%%%%%%%%%%%%%%%%%%%%%%%%%%%%%%%%%%%%%%%%%%%%%%%%%%%%%%%%%%%%%%%%%%%%%
\section{Effective resistance and complexity \label{sec05}} 

In this section, we derive formulae for effective resistance and spanning trees, denoting the effective resistance between two vertices $x$ and $y$ in a graph $G$ as $r(G;x,y)$. 
Nash-Williams \cite{Nash} proved the following surprising formula between effective resistance and average hitting time. 

\begin{theorem}[Nash-Williams \cite{Nash}] \label{Nash}
\begin{align*}
r(G;x,y)=\displaystyle\frac{h(G;x,y)+h(G;y,x)}{2|E(G)|}.
\end{align*}
\end{theorem}

By combining the results of Theorems~\ref{Mainthm} and \ref{Nash}, we can obtain the exact formula for effective resistance in the %odd-vertex 
wheel graph $W_{N+1}$ as follows. 

\begin{lemma} \label{lem1}
The exact formula for the effective resistance between periphery vertices $0$ and $\ell$ $(1\leq \ell\leq N-1)$ in $W_{N+1}$ is 
\begin{align*}
r(W_{N+1}; 0, \ell)
 =\begin{cases}
\displaystyle\frac{2(F_{N}-F_{N-2\ell})}{F_{N-1}+F_{N+1}} & \textit{if} \ N \text{ is odd,}\\
\displaystyle\frac{2(L_{N}-L_{N-2\ell})}{L_{N-1}+L_{N+1}} & \textit{if} \ N \text{ is even.}
\end{cases}
\end{align*}
\end{lemma}

\begin{proof}
If $N$ is odd, then
\begin{align*}
r(W_{N+1}; 0, \ell)
 & = \frac{h(W_{N+1};0,\ell)+h(W_{N+1};\ell,0)}{2|E(W_{N+1})|}\\
 & = \frac{h(W_{N+1};0,\ell)}{|E(W_{N+1})|}\\
 & =\frac{1}{2N}\cdot\frac{4N(F_{N}-F_{N-2\ell })}{F_{N-1}+F_{N+1}}\\
 & =\frac{2( F_{N} -F_{N-2\ell })}{F_{N-1} +F_{N+1}}.
\end{align*}

If $N$ is even, then
\begin{align*}
r(W_{N+1}; 0, \ell)
 & = \frac{h(W_{N+1};0,\ell)+h(W_{N+1};\ell,0)}{2|E(W_{N+1})|}\\
 & = \frac{h(W_{N+1};0,\ell)}{|E(W_{N+1})|}\\
 & =\frac{1}{2N}\cdot\frac{4N( L_{N} -L_{N-2\ell })}{L_{N-1} +L_{N+1}}\\
 & =\frac{2( L_{N} -L_{N-2\ell })}{L_{N-1} +L_{N+1}}.
\end{align*}
\end{proof}

By combining the results of Theorems~\ref{Mainthm} and \ref{Nash}, we can compute the effective resistance in the wheel graph $W_{N+1}$ as follows. 

\begin{lemma} \label{lem2}
The exact formula for the effective resistance between center vertex $N$ and periphery vertex $0$ in $W_{N+1}$ is
\begin{align*}
r(W_{N+1}; N, 0)
 =\begin{cases}
\displaystyle\frac{F_{N}}{F_{N-1}+F_{N+1}} & \textit{if} \ N \text{ is odd,}\\
\displaystyle\frac{L_{N}}{L_{N-1} +L_{N+1}} & \textit{if} \ N \text{ is even.}
\end{cases}
\end{align*}
\end{lemma}

\begin{proof}
If $N$ is odd, then
\begin{align*}
r(W_{N+1}; N, 0)
 & = \frac{h(W_{N+1};N,0)+h(W_{N+1};0,N)}{2|E(W_{N+1})|}\\
 & =\frac{1}{4N}\left(\frac{(4N-3)F_{N+1}-(4N+3)F_{N-1}}{F_{N-1}+F_{N+1}}+3\right)\\
 & =\frac{1}{4N}\cdot\frac{(4N-3)F_{N+1}-(4N+3)F_{N-1}+3(F_{N-1}+F_{N+1})}{F_{N-1}+F_{N+1}}\\
 & =\frac{1}{4N}\cdot\frac{4N(F_{N+1}-F_{N-1})}{F_{N-1}+F_{N+1}}\\
 & =\frac{F_{N}}{F_{N-1}+F_{N+1}}.
\end{align*}

If $N$ is even, then
\begin{align*}
r(W_{N+1}; N, 0)
 & = \frac{h(W_{N+1};N,0)+h(W_{N+1};0,N)}{2|E(W_{N+1})|}\\
 & =\frac{1}{4N}\left(\frac{( 4N-3) L_{N+1} -( 4N+3) L_{N-1}}{L_{N-1} +L_{N+1}}+3\right)\\
 & =\frac{1}{4N}\cdot\frac{(4N-3) L_{N+1} -( 4N+3) L_{N-1}+3( L_{N+1} +L_{N-1}) }{L_{N-1} +L_{N+1}}\\
 & =\frac{1}{4N}\cdot\frac{4N( L_{N+1} -L_{N-1})}{L_{N-1} +L_{N+1}}\\
 & =\frac{L_{N}}{L_{N-1} +L_{N+1}}.
\end{align*}
\end{proof}

Denoted by $T(G)$, the {\textit{graph complexity}} of a graph $G$ is the number of spanning trees in $G$. Then let $\tau(G; u, v)$ denote the graph complexity of the graph obtained by identifying two vertices $u$ and $v$ of $G$. Kirchhoff~\cite{Kir} proved the following well-known formula relating effective resistance and graph complexity. 

\begin{theorem}[Kirchhoff \cite{Kir}] \label{Kir}
\begin{align*}
\tau(G; u, v)=r(G; u, v) \cdot T(G). 
\end{align*}
\end{theorem}

Also derived by Myers~\cite{Myers}, Sedlacek~\cite{Sedlacek} proved the following formula for the graph complexity of $W_{N+1}$ using the matrix tree theorem~\cite{Kir}.

\begin{theorem}[Myers~\cite{Myers}, Sedlacek~\cite{Sedlacek}]\label{Myers}
\begin{align*}
T(W_{N+1})=L_{2N}-2. 
\end{align*}
\end{theorem}

The number of spanning trees in $W_{N+1}$ when the peripheral vertices are identified can be computed as follows by combining Theorems~\ref{Kir} and \ref{Myers} and Lemma~\ref{lem1}. 

%%%%%  定  理  %%%%%
\begin{theorem} \label{thm7}
The exact %general 
formula for the number of spanning trees when the peripheral vertices of $W_{N+1}$ are identified, denoted as $\tau(W_{N+1}; 0, \ell)$ $(1\leq \ell\leq N-1)$, can be written as follows:
%For $1\leq \ell\leq N-1$,
\begin{align*}
\tau(W_{N+1}; 0, \ell)
 =\begin{cases}
2(F_{N}-F_{N-2\ell })( F_{N-1} +F_{N+1}) & \text{if } N\text{ is odd,}\\
\displaystyle\frac{2(L_{N}-L_{N-2\ell})(L_{N}-2)(L_{N}+2)}{L_{N-1}+L_{N+1}} & \text{if } N\text{ is even.}
\end{cases}
\end{align*}
\end{theorem}

If $N$ is odd, then
\begin{align*}
\tau(W_{N+1}; 0, \ell)
 & =r(W_{N+1}; 0, \ell) \cdot T(W_{N+1})\\
 & =\frac{2( F_{N} -F_{N-2\ell })}{F_{N-1} +F_{N+1}} ( L_{2N} -2)\\
 & =\frac{2( F_{N} -F_{N-2\ell })}{F_{N-1} +F_{N+1}} (  F_{2N-1} +F_{2N+1} -2)\\
 & =\frac{2( F_{N} -F_{N-2\ell })}{F_{N-1} +F_{N+1}} ( F_{N}^{2} +F_{N-1}^{2} +F_{N+1}^{2} +F_{N}^{2} -2) &&\text{[by \eqref{Fib11}]}\\
 & =\frac{2( F_{N} -F_{N-2\ell })}{F_{N-1} +F_{N+1}} (( F_{N-1} +F_{N+1})^{2} +2( F_{N}^{2} -F_{N-1} F_{N+1}) -2)\\
 & =\frac{2( F_{N} -F_{N-2\ell })}{F_{N-1} +F_{N+1}} (( F_{N-1} +F_{N+1})^{2} +2\cdot ( -1)^{N+1} -2) &&\text{[by \eqref{Fib10}]}\\
 & =2( F_{N} -F_{N-2\ell })( F_{N-1} +F_{N+1}).
\end{align*}

If $N$ is even, then
\begin{align*}
\tau(W_{N+1}; 0, \ell)
 & =r( W_{N+1} ;0,\ell ) \cdot T( W_{N+1})\\
 & =\frac{2( L_{N} -L_{N-2\ell })}{L_{N-1} +L_{N+1}} ( L_{2N} -2)\\
 & =\frac{2( L_{N} -L_{N-2\ell })}{L_{N-1} +L_{N+1}} (L_N^2-2\cdot(-1)^N-2) &&\text{[by \eqref{Luc10}]}\\
 & =\frac{2( L_{N} -L_{N-2\ell })( L_{N} -2)( L_{N} +2)}{L_{N-1} +L_{N+1}}.
\end{align*}

Finally, the number of spanning trees when the center and the peripheral vertices of $W_{N+1}$ are identified can be expressed as follows: 

%%%%%  定  理  %%%%%
\begin{theorem} \label{thm8}
The exact formula for the number of spanning trees when the center and peripheral vertices of $W_{N+1}$ are identified, denoted as $\tau(W_{N+1}; N, 0)$, can be written as follows: 
\begin{align*}
\tau(W_{N+1};N,0)
 =\begin{cases}
F_N( F_{N-1} +F_{N+1}) & \text{if } N\text{ is odd}\\
\displaystyle\frac{L_N( L_{N} -2)( L_{N} +2)}{L_{N-1} +L_{N+1}} & \text{if } N\text{ is even}
\end{cases}
\end{align*}
\end{theorem}

If $N$ is odd, then
\begin{align*}
\tau (W_{N+1}; N,0)
 & =r(W_{N+1}; N,0) \cdot T( W_{N+1})\\
 & =\frac{F_{N}}{F_{N-1}+F_{N+1}}( L_{2N} -2)\\
 & =\frac{F_{N}}{F_{N-1}+F_{N+1}}(F_{2N-1} +F_{2N+1} -2)\\
 & =\frac{F_{N}}{F_{N-1} +F_{N+1}} ( F_{N}^{2} +F_{N-1}^{2} +F_{N+1}^{2} +F_{N}^{2} -2) &&\text{[by \eqref{Fib11}]}\\
 & =\frac{F_{N}}{F_{N-1} +F_{N+1}} (( F_{N-1} +F_{N+1})^{2} +2( F_{N}^{2} -F_{N-1} F_{N+1}) -2)\\
 & =\frac{F_{N}}{F_{N-1} +F_{N+1}} (( F_{N-1} +F_{N+1})^{2} +2\cdot ( -1)^{N+1} -2) &&\text{[by \eqref{Fib10}]}\\
 & =F_N(F_{N-1}+F_{N+1}).
\end{align*}

If $N$ is even, then
\begin{align*}
\tau ( W_{N+1} ;N,0 ) & =r( W_{N+1} ;N,0 ) \cdot T( W_{N+1})\\
 & =\frac{L_{N}}{L_{N-1} +L_{N+1}}( L_{2N} -2)\\
 & =\frac{L_{N}}{L_{N-1} +L_{N+1}}(L_N^2-2\cdot(-1)^N-2) &&\text{[by \eqref{Luc10}]}\\
 & =\frac{L_{N}( L_{N} -2)( L_{N} +2)}{L_{N-1} +L_{N+1}}.
\end{align*}

%%%%%%%%%%%%%%%%%%%%%%%%%%%%%%%%%%%%%%%%%%%%%%%%%%%%%%%%%%%%%%%%%%%%%%%%%%%%%%%%%%%%%%%%%%%%%%%%%%%%%%%%%%%%%%%%%%%%%%%%%%%%%%%%%%%%%%%%%%%%%%%%%%%%%%%%%%
\section{Conclusion \label{sec06}} 
In this paper, by using an approach similar to that in \cite{Etal, TS, TY, TY2}, we provided exact formulas for the average hitting times of simple random walks on the wheel graph $W_{N+1}$, where the cycle part consists of $N$ vertices and a single vertex is placed at the center of the cycle. We showed that these formulas can be expressed in terms of the Fibonacci and Lucas numbers, depending on whether the number of vertices is odd or even. 
Using these formulas, we also derived explicit expressions for the key graph-theoretic metrics of the effective resistance and the number of spanning trees when two vertices are identified. 

Interesting problems for future research would be extending our model to include multiple internal vertices and investigating the average hitting times in such configurations, as well as exploring the underlying reasons why Fibonacci and Lucas numbers appear in these formulas.

\section*{Acknowledgments}
The research of Yuuho Tanaka is supported by a JSPS Grant-in-Aid for JSPS Fellows (23KJ2020).

%%%%%%%%%%%%%%%%%%%%%%%%%%%%%%%


\begin{thebibliography}{99}
\bibitem{Ald}
Aldous, D. J. : 
Hitting times for random walks on vertex-transitive graphs. 
\textit{Mathematical Proceedings of the Cambridge Philosophical Society} {\bf 106} 179--191 (1989)

%\bibitem{Chair}
%Chair, N. :
%The effective resistance of the $N$-cycle graph with four nearest neighbors,
%\textit{J. Stat. Phys.} {\bf154}, (2014), 1177--1190.

\bibitem{Etal} 
Doi, Y., Konno, N., Nakamigawa, T., Sakuma, T., Segawa, E., Shinohara, H., Tamura, S., Tanaka, Y.,  and Toyota, K. :
On the average hitting times of the squares of cycles. 
\textit{Discrete Applied Mathematics} {\bf 313} 18-28  (2022)

\bibitem{Kir}
Kirchhoff, G. : 
Uber die Aufl\"{o}sung der Gleichungen, auf welche man bei der Untersuchung der linearen Verteilung galvanischer Str\"{o}me gef\"{u}hrt wird, \textit{Annalen Physik und Chemie} \textbf{72} 497-508  (1847)

%\bibitem{LYZ}
%Li, C., Yang, Y., Zhang, H. :
%Kirchhoff index of weighted wheel graphs $W_n(a,b)$,
%\textit{Journal of Lanzhou University} {\bf 44}(3), (2008), 100--102.

\bibitem{Lovasz}
Lov$\acute{a}$sz, L. : Random walks on graphs: A survey. 
\textit{Combinatorics, Paul Erdős is eighty} Vol. 2 (1993)

\bibitem{Myers}
Myers, B. R. : 
Number of spanning trees in a wheel. 
\textit{IEEE Transactions on Circuit Theory} {\bf CT-18}, 280-282 (1971)

\bibitem{Nash}  C. S. J. A. Nash-Williams. : 
Random walk and electric currents in networks. \textit{Proceedings of the Cambridge Philosophical Society. Mathematical and Physical Sciences} \textbf{55} 181-194 (1959)

\bibitem{Nishimura}
Y. Nishimura. : 
Average hitting times in some $f$-equitable graphs,
https://arxiv.org/abs/2312.11848

\bibitem{Sharavas}
Sharavas, K. R. : 
Finding hitting times in various graphs. 
\textit{Statistics \& Probability Letters} {\bf 83} 2067-2072 (2013)

\bibitem{Sedlacek}
Sedlacek, J. : 
On the skeletons of a graph or digraph. 
\textit{Proceedings of the Calgary International Conference of Combinatorial Structures and Their Applications}, Gordon and Breach, 387-391 (1970)

\bibitem{TS}
Tamura, S. : 
On the average hitting times of the directed wheel, 
in submitted. 

\bibitem{TY}
Tanaka, Y. : 
On the average hitting times of $\mbox{Cay}(\mathbb{Z}_N, \{+1, +2\})$. 
\textit{Discrete Applied Mathematics} {\bf 343} 269-276 (2024)

\bibitem{TY2}
Tanaka, Y. : 
On the average hitting times of weighted Cayley graphs,
https://arxiv.org/abs/2310.16571

\bibitem{Thomas}
Thomas, K. : \textit{Fibonacci and Lucas Numbers with Applications, Volume 1}, Pure and Applied Mathematics: A Wiley Series of Texts, Monographs and Tracts (2018)

\bibitem{Koshy}
Thomas, K. : \textit{Fibonacci and Lucas Numbers with Applications, Volume 2}, Pure and Applied Mathematics: A Wiley Series of Texts, Monographs and Tracts (2019)

%\bibitem{Tutte}
%Tutte, W. T. :
%The dissection of equilateral triangles into equilateral triangles. 
%\textit{Math. Proc. Cambridge Philos.Soc.} {\bf 44} 463-482 (1948)

%\bibitem{Gabow}
%Gabow, H. N. and Myers, E. W. : 
%Finding all spanning trees of directed and undirected graphs. 
%\textit{SIAM Journal on Computing}, \textbf{7} 280-287 (1978)

%\bibitem{Louis}
%Justine, L. : 
%Low temperature ratchet current. 
%\textit{Mathematics and Mechanics of Complex Systems} vol.7 no.3 (2019)

\bibitem{YY}
Yang, Y. : 
Simple random walks on wheel graphs. 
\textit{Applied Mathematics $\&$ Information Sciences} {\bf 6} 123-128 (2012)

%\bibitem{new}
%Tamura, S. and Tanaka, Y.: 
%in preparation. 

\end{thebibliography}
\end{document}